\documentclass[11pt]{amsart}
\usepackage{amsmath,amsthm,amsfonts,amssymb}

\usepackage{color}

\newtheorem{thm}[equation]{Theorem}
\newtheorem{cor}[equation]{Corollary}
\newtheorem{lem}[equation]{Lemma}
\newtheorem{prop}[equation]{Proposition}
\theoremstyle{definition}

\numberwithin{equation}{section}
\newcommand{\fgl}{\mathfrak{gl}}
\newcommand{\fsl}{\mathfrak{sl}}

\newcommand{\gl}{\mathfrak{gl}}

\newcommand{\fosp}{\mathfrak{osp}}
\newcommand{\fpsl}{\mathfrak{psl}}

\newcommand{\fa}{\mathfrak{fa}}
\definecolor{mjo}{rgb}{0,0,.9}

\newcommand{\complex}{\mathbb{C}}

\newcommand{\Z}{\mathbb{Z}}

\newcommand{\fs}{\mathfrak{s}}
\newcommand{\cc}{\mathbb{C}}

\newcommand{\z}{\mathbb{Z}}

\newcommand{\fg}{\mathfrak{g}}

\newcommand{\fr}{\mathfrak{r}}

\newcommand{\fm}{\mathfrak{m}}
\newcommand{\g}{\mathfrak{g}}

\newcommand{\h}{\mathfrak{h}}
\newcommand{\fp}{\mathfrak{p}}

\newcommand{\Ind}{\mbox{Ind}}
\newcommand{\n}{\mathfrak{n}}

\newcommand{\lm}{\mathcal{L}}
\newcommand{\W}{\mathcal{W}}
\newcommand{\F}{\mathcal{F}}

\DeclareMathOperator{\Wh}{Wh}
\newcommand{\M}{\widetilde{M}}

\newcommand{\0}{\bar 0}
\newcommand{\1}{\bar 1}
\DeclareMathOperator{\ad}{ad}

\def \fa{\mathfrak{a}}

\begin{document}

\title[Whittaker modules for Lie superalgebras]{Whittaker categories and strongly typical Whittaker modules for Lie superalgebras}

\author{Irfan Bagci}
\address{Department of Mathematics \\
           University of California, Riverside\\ Riverside, CA 92521}
           \email{irfan@math.ucr.edu}          
\author{Konstantina Christodoulopoulou}
\address{Department of Mathematics \\
           University of Connecticut\\ Storrs, CT 06269}
\email{christod@uconn.edu}
\author{Emilie Wiesner}
\address{Department of Mathematics \\
           Ithaca College\\ Ithaca, NY 14850  }
\email{ewiesner@ithaca.edu}

\maketitle

\begin{abstract}
Following analogous constructions for Lie algebras, we define Whittaker modules and Whittaker categories for finite-dimensional simple Lie superalgebras.  Results include a decomposition of Whittaker categories for a Lie superalgebra according to the action of an appropriate sub-superalgebra; and, for basic classical Lie superalgebras of type I, a description of the strongly typical simple Whittaker modules.
\end{abstract}

\section{Introduction}

``Whittaker" modules and Whittaker categories have been studied for a variety of Lie algebras and have a well-developed theory in the Lie algebra setting.  This work relies on Lie algebra structures that have natural analogues for Lie superalgebras.  Here we focus on defining and investigating Lie superalgebra equivalents of both Whittaker modules and Whittaker categories.

Kostant \cite{Ko1978} defined Whittaker modules for complex finite-dimensional semisimple Lie algebras $\g$.  Such a Lie algebra has a triangular decomposition $\g= \n^- \oplus \h \oplus \n^+$ (where $\h$ is a Cartan subalgebra and $\n^{\pm}$ is nilpotent); a Whittaker module is a $\g$-module $V$ generated (as a $\g$-module) by a one-dimensional $\n^+$-submodule $\cc w \subseteq V$.  Among other results, Kostant provided a description of Whittaker modules $V$, for a restricted set of generating $\n^+$-submodules $\cc w$, in terms of the action of the center $Z(\g)$ of $U(\g)$.     McDowell \cite{mcdowell} and Mili$\check{\hbox{\rm c}}$i$\acute{\hbox{\rm c}}$ and Soergel \cite{milicic-soergel} built on Kostant's work to produce a description of Whittaker modules for all generating $\n^+$-submodules $\cc w$, along with results that situated Whittaker modules in a larger category of $\g$-modules with locally finite $\n^+$- and $Z(\g)$-actions.

Many other Lie algebras and related algebras possess a structure similar to the triangular decomposition of complex finite-dimensional semisimple Lie algebras, and there has been a variety of work done to define and investigate Whittaker modules in these settings.  Whittaker modules have been studied for Heisenberg Lie algebras \cite{Ch08}, for the Virasoro algebra \cite{OW2008}, for the twisted Heisenberg-Virasoro algebra \cite{Li10}, and for graded Lie algebras \cite{Wa10}. There has also been work on Whittaker modules for quantum groups \cite{Se00,On05} and for generalized Weyl algebras \cite{BO08}.

More recently, Batra and Mazorchuk developed a framework to unite and expand on the previous appearances of Whittaker modules.  They considered pairs of Lie algebras $\g \supseteq \n$ where $\n$ is a nilpotent subalgebra of $\g$ and investigated the category $\W$ of $\g$-modules such that $\n$ acts locally finitely.  (We refer to this category as a Whittaker category.) Among their main results in this general setting, they showed that if the action of $\n$ on $\g$ is locally finite, then the category $\W$ decomposes according to the action of $\n$.

In this paper we adapt the framework of Batra and Mazorchuk to finite-dimensional simple Lie superalgebras.  Section \ref{section:background} presents relevant background information for Lie superalgebras.  Section \ref{section:WhittakerCategories} defines Whittaker categories $\W$ for Lie superalgebra pairs $(\g, \n)$ and gives a decomposition of $\W$ according to the action of $\n$, among other results.  

Unlike the Lie algebra setting, simple finite-dimensional modules for a finite-dimensional nilpotent Lie superalgebra are not always one-dimensional; this creates an additional challenge for reproducing Lie algebra results in the Lie superalgebra setting.  For this reason, we restrict to basic classical Lie superalgebras of type I in Sections  \ref{section:TypeI} and \ref{section:nonsingular}.  In Section  \ref{section:TypeI}, we construct modules $\widetilde{M}_{\psi, \lambda}$ ($\psi \in\lm\subseteq\n^*$, defined below, and $\lambda \in \h^*$) in the Lie superalgebra category $\W(\g, \n)$ by inducing up from particular modules  $M_{\psi, \lambda}$ in the underlying Lie algebra category $\W(\g_{\overline 0}, \n_{\overline 0})$.  We show that, under certain mild restrictions on $\psi$ and $\g$,  the modules $\widetilde{M}_{\psi, \lambda}$ have unique simple quotients $\widetilde{L}_{\psi, \lambda}$. Moreover, for suitable choices of $\lambda$, the modules  $\widetilde{L}_{\psi, \lambda}$ give a complete list up to isomorphism of the strongly typical Whittaker modules. 

If $\psi: \n \rightarrow \cc$ is nonzero on the simple roots of $\g_{\overline{0}}$, then the $\g_{\overline{0}}$-module $M_{\psi, \lambda}$ is simple.  In Section \ref{section:nonsingular} we study the corresponding induced modules $\widetilde{M}_{\psi, \lambda}$.  In particular, we show that the $\fsl(1,2)$-modules $\widetilde{M}_{\psi, \lambda}$ are simple and thus give a complete description of the strongly typical simple Whittaker $\fsl(1,2)$-modules.

\section{Background for Lie Superalgebras} \label{section:background}
In this paper we restrict to finite-dimensional Lie superalgebras over $\cc$.  Such a Lie superalgebra is a $\z_2$-graded vector space $\g=\g_{\overline{0}} \oplus \g_{\overline{1}}$ with a bracket $[ , ]: \g \times \g  \rightarrow \g$ which preserves the $\z_2$-grading and satisfies graded versions of the operations used to define Lie algebras.  Let $d(x)$ denote the $\z_2$-degree of a homogeneous element $x \in \g$. The even part $\g_{\overline{0}}$ is a Lie algebra under the bracket operation. Finite-dimensional complex simple Lie superalgebras were classified by Kac \cite{Kac}.

In this paper $\g$-modules $V$ are $\z_2$-graded: $V= V_{\overline{0}} \oplus V_{\overline{1}}$.  Thus there is a parity change functor $\Pi$ on the category of $\g$-modules, which interchanges the $\z_2$-grading of a module.  For a finite-dimensional module $V$ and a simple module $L$, define $[V:L]$ to be the number of times that $L$ appears as a factor in a composition series of $V$.  We restrict to modules $V$ with at most a countable basis.

For a  Lie (super)algebra $\mathfrak{a}$ denote by $U(\mathfrak{a})$ its universal enveloping (super)algebra and by  $Z(\mathfrak{a})$ the center of $U(\mathfrak{a})$. We call an algebra homomorphism $\chi: Z(\fa) \to\cc$ a character of $Z(\fa)$. We
say that an $\fa$-module $V$ has central character $\chi$ if $zv = \chi(z) v$ for every $z \in Z(\fa)$ and $v\in V$.

 Note that $U(\fg)$ can be viewed as an $\fg$-module through the adjoint action:
\begin{equation}\label{def:adjoint action on U(g)}
(ad x)u=xu-(-1)^{d(x)d(u)}ux
\end{equation} for homogeneous $x\in \fg$, $u\in U(\fg)$. 

The superalgebra $U(\fg)$ is a  supercommutative Hopf superalgebra.   Thus we can define a module structure on a tensor product of modules using the coproduct.  For  a sub-superalgebra $\n$ of $\fg$, we define induced modules as follows.  First note that the Hopf superalgebra structure on $U(\fg)$ gives a $\g$-module structure on the tensor product $M\otimes N=:M\otimes_{\cc}N$, where $M,N$ are $\fg$-modules.
Then  for an $\n$-module $N$, $\Ind_\n^\fg N=U(\fg)\otimes_{U(\n)}N$ has the natural structure of a left $U(\fg)$-module ($U(\fg)$ is considered as a right $U(\n)$-module and a left $U(\fg)$-module through multiplication).


\subsection{Nilpotent finite-dimensional Lie superalgebras}
Let $\n=\n_{\overline{0}} \oplus \n_{\overline{1}}$ be a nilpotent Lie superalgebra, and let $\F(\n)$ (abbreviated as $\F$) be the category of finite-dimensional $\n$-modules. 

Define 
\begin{equation}
\lm = \{ \psi \in \n^* \mid \psi(\n_{\overline{1}})=0, \psi([\n_{\overline{0}},\n_{\overline{0}}])=0 \}.
\end{equation}

For a fixed $\psi \in \lm$, one can choose $\h \subseteq \n$ such that 
\begin{itemize}
\item $\n_{\overline{0}} \subseteq \h$;
\item $\psi([\h,\h])=0$;
\item $\h_{\overline 1}$ is maximal so that $\psi([\h_{\overline 1},\h_{\overline 1}])=0$.
\end{itemize}
Thus, $\h$ is maximal in $\n$ so that $\psi\mid_{\h}$ is a Lie superalgebra homomorphism. Define
$$
I(\psi) = \Ind_{\h}^{\n} (\psi) = U(\n) \otimes_{U(\h)} \cc_{\psi}
$$ 
 Note that if $\psi([\n_{\overline{1}}, \n_{\overline{1}}])=0$, then $\h=\n$ and $I(\psi)$ is one-dimensional. 

The following is a summary of several results from \cite{Kac}, with a correction given in \cite{Ser}.  (These results are stated for solvable Lie superalgebras.  When applied to nilpotent Lie superalgebras, the results in the two papers coincide.) 
\begin{prop} \label{lem:n irreducible modules}
For $\psi \in \lm$, the module $I(\psi)$ is independent of the choice of $\h$, is finite-dimensional, and is irreducible.  Moreover, up to the parity functor, the set $\{I(\psi) \mid \psi \in \lm \}$ provides a complete list of the irreducible objects in $\F$. 
\end{prop}

In this paper it will also be useful to view $I(\psi)$ as an $\n_{\overline{0}}$-module.  The following lemma shows that, as an $\n_{\overline{0}}$-module, $I(\psi)$ has a composition series with all factors isomorphic to $\cc_{\psi}$.
\begin{lem} \label{lem;n0nilpotent}
Let $\psi \in \lm$.  Then the $\n_{\overline{0}}$-action $x-\psi(x)$ on $I(\psi)$ is locally nilpotent.
\end{lem}

\begin{proof}
Consider $I(\psi)$.
The action is well-defined since $\psi([\n_{\overline{0}},\n_{\overline{0}}])=0$.

To show that the action is locally nilpotent on $I(\psi)$, consider a monomial $y \otimes 1_{\psi} \in I(\psi)$.  Then,
$$(x- \psi(x))(y \otimes 1_{\psi}) = [x,y] \otimes 1_{\psi} + y \otimes (x- \psi(x))1_{\psi}= [x,y] \otimes 1_{\psi} .$$
Since the induced adjoint action of $\n_{\overline{0}}$ on $U(\n)$ is locally nilpotent, it follows that for $k$ sufficiently large
$$
(x- \psi(x))^k(y \otimes 1_{\psi})= [x,y]^k\otimes 1_{\psi} =0.
$$
\end{proof}

\begin{cor} \label{cor:finitedecomp}
Let $V \in \F$.  Then $$V= \bigoplus_{\psi \in \lm} V^{\psi},$$ where $V^{\psi}$ is the maximal submodule such that, for any irreducible module $L$, $[V^{\psi}: L] \neq 0$ only if $L \cong I(\psi), \Pi(I(\psi))$.
\end{cor}
\begin{proof}
Viewed as an $\n_{\overline{0}}$-module, $\displaystyle V=\bigoplus_{\psi \in \lm} V^{\psi}$ where 

$$V^{\psi}=\{v \in V \mid \mbox{for $x\in \n_{\overline{0}}$, $(x-\psi(x))^kv=0$ for $k>>0$}\}.$$
Similar to the argument in the previous proof,  for $v \in V^{\psi}$, $y \in \n_{\overline{1}}$, and $x \in \n_{\overline{0}}$, we have that $(x-\psi(x))^kyv=0$ for $k$ sufficiently large.  Thus, $V^{\psi}$ is in fact a $\n$-submodule of $V$.  The result then follows from Lemma \ref{lem:n irreducible modules}.
\end{proof}

Based on this decomposition, we define $\F(\psi)$ (or $ \F(\n, \psi)$) to be the full subcategory of $\F$ such that $V \in \F(\psi)$ implies $[V: L]=0$ unless $L \cong I(\psi), \Pi(I(\psi))$.

\section{Whittaker categories for Lie superalgebras} \label{section:WhittakerCategories}
In this section, we define Whittaker categories for Lie superalgebras and present a category decomposition.  
Let $\g$ be a finite-dimensional Lie superalgebra and $\n \subseteq \g$ a nilpotent sub-superalgebra.  Following the definition of Batra and Mazorchuk \cite{BM2010}, we define a {\it Whittaker category} $\W(\g, \n)$ for the pair $(\g, \n)$ to be the full category containing $\g$-modules $V$ such that $\n$ acts locally finitely on $V$. When there is no confusion, we may abbreviate $\W(\g, \n)$ as $\W$.

Let $V \in \F(\n)$.  Since the adjoint action of $\n$ is locally finite on $U(\g)$, ${\rm Ind}_{\n}^{\g}(V) \in \W(\g, \n)$. For any $M \in \W(\g, \n)$, $M$ has a filtration by finite-dimensional $\n$-modules: $0=M_0 \subseteq M_1 \subseteq \cdots$, $\bigcup_i M_i =M$.  (To see this, recall that we only consider modules with a countable basis.) Then, for an irreducible $\n$-module $L$, define
$$
[M: L]= \mbox{sup}_{i \geq 0} \{[M_i:L] \}.
$$  


\begin{lem}
Fix $\psi \in \lm$, and let $V \in \F(\n, \psi)$.  View ${\rm Ind}_{\n}^{\g} V$ as an $\n$-module; then for any irreducible $\n$-module $L$, $[{\rm Ind}_{\n}^{\g} V: L] = 0$ unless $L \cong I(\psi), \Pi(I(\psi))$.
\end{lem}
\begin{proof}
View ${\rm Ind}_{\n}^{\g} V$ as an $\n_{\overline 0}$-module with action given by $x . v= (x-\psi(x)) v$. It follows from Lemma \ref{lem;n0nilpotent} that the action of $\n_{\overline{0}}$ on $V$ is locally nilpotent.  Since the adjoint action of $\n_{\overline 0}$ on $U(\g)$ is locally finite, a similar argument to the proof of Lemma \ref{lem;n0nilpotent} shows that the action of $\n_{\overline 0}$ on all of ${\rm Ind}_{\n}^{\g} V$ is locally nilpotent.

Now suppose that $L$ is an irreducible subquotient of ${\rm Ind}_{\n}^{\g} V$. We know that for $x \in \n_{\overline 0}$, $(x-\psi(x))$ acts nilpotently on $L$.  Since $\psi$ determines an irreducible $\n$-module up to $\Pi$, it must be that $L \cong I(\psi)$ or $L \cong \Pi(I(\psi))$.  
\end{proof}

\begin{thm} \label{theorem: category decomposition}
Let $M \in \W(\g, \n)$.  Then,
$$
M=\bigoplus_{\psi \in \lm} M^{\psi},
$$
where $M^{\psi}$ is the maximal $\g$-submodule of $M$ such that, for an irreducible $\n$-module $L$, $[M:L]=0$ unless $L \cong I(\psi), \Pi(I(\psi))$.  

\end{thm}

\begin{proof}
Define $M^{\psi}$ to be the sum of the images of all modules ${\rm Ind}_{\n}^{\g}V$ where $V \in \F(\n, \psi)$.  Then $M^{\psi}$ is a submodule of $M$, since it is a sum of submodules.   Moreover, for any $v \in M^{\psi}$ and $x \in \n_{\overline 0}$, $(x-\psi(x))^kv=0$ for $k>>0$.  Since sums of distinct generalized eigenspaces are direct, this implies that $\sum_{\psi \in \lm} M^{\psi}$ is a direct sum.

Now we argue that $\bigoplus_{\psi \in \lm}M^{\psi} =M$.  Let $v \in M$ and define $V = U(\n)v$.  By definition of $\W(\g, \n)$, $V$ is finite-dimensional and so by Corollary \ref {cor:finitedecomp} $V = \bigoplus_{\psi \in \lm}V^{\psi}$.  By definition of $M^{\psi}$, $V^{\psi} \subseteq M^{\psi}$.  This implies that $v \in V \subseteq  \bigoplus_{\psi \in \lm}M^{\psi}$, which completes the proof.
\end{proof}

Based on this result, we define $\W(\g, \n, \psi)$ (or $\W(\psi)$) as the full subcategory of $\W(\g, \n)$ where $M \in \W(\g, \n, \psi)$ implies that, for an irreducible $\n$-module $L$, $[M:L]=0$ unless $L \cong I(\psi), \Pi(I(\psi))$. 

\begin{lem}\label{lem: any module in Whittaker cat contains Whittaker vector}
Let $\psi \in \lm$ and $V \in \W(\psi)$.  Viewing $V$ as an $\n$-module,  $V$ contains a submodule isomorphic to $I(\psi)$ or $\Pi(I(\psi))$.
\end{lem}

\begin{proof}
Let $0 \neq v \in V$.  By definition of the category, $U(\n)v$ is a finite-dimensional $\n$-module.  Therefore, $U(\n)v$ has a nonzero socle which must be a direct sum of $I(\psi), \Pi(I(\psi))$. 
\end{proof}

Suppose $\psi \in \lm$ so that $\psi([\n_{\overline{1}},\n_{\overline{1}}])=0$ (and thus $I(\psi)$ is one-dimensional).  For $V \in\W(\psi)$  define
\begin{equation} \label{defn:space of Whittaker vectors for psi}
\Wh_{\psi}(V)=\{v\in V\mid xv=\psi (x) v \textrm{ for all } x\in\n \}. \end{equation}

The vectors in $\Wh_{\psi}(V)$ are the Whittaker vectors of $V$.  (Thus, if $w\in \Wh_{\psi}(V),$ then $\cc w$ is a one-dimensional irreducible $\n$-module.) If $V$ is generated by a Whittaker vector $w$, then we call $V$ a {\it Whittaker module}.

\begin{cor}\label{cor:Whittaker module with 1-dim'l space of Whittaker vectors is irreducible}
Let $\psi \in \lm$ so that $\psi([\n_{\overline{1}},\n_{\overline{1}}])=0$ and let $V \in \W(\psi)$ be a Whittaker module. If $\dim \Wh_\psi(V)=1$, then $V$ is irreducible.
\end{cor}

\begin{proof}
Let $W$ be a non-trivial  $\g$-submodule of $V$. Lemma  \ref{lem: any module in Whittaker cat contains Whittaker vector} implies that $W$ contains a non-zero Whittaker vector $w$.  Since $\dim \Wh_\psi(V)=1$, it follows that $w$ must generate $V$. 
\end{proof}

\section{Whittaker modules for basic classical Lie superalgebras of type I}\label{section:TypeI}
For the remainder of the paper we restrict $\g$ to a simple basic classical Lie superalgebra of type I,  that is $\g=\fsl(m,n)$ $(n>m\geq1)$,  $\fpsl(n,n)$ $(n\geq3)$, or $\g=\fosp(2,2n)$ $(n\geq1)$. Although we will not consider this case here, our results can easily be extended to $\g=\fgl(m,n)$. For a natural choice of nilpotent subalgebra $\n \subseteq \g$, the irreducible $\n$-modules are one-dimensional.  With this in mind, we study Whittaker modules and Whittaker vectors for these Lie superalgebras.  

\subsection{Basic Classical Lie superalgebras of Type I} \label{section: Basic classical Lie algebras of type I}
The basic classical Lie superalgebras of type I admit a $\Z$-grading $\g=\g_{-1}\oplus\g_0\oplus\g_1$, where $\g_{\0}=\g_0$ and $\g_{\1}=\g_{-1}+\g_1$,  and a non-degenerate even invariant supersymmetric bilinear form $(\cdot,\cdot)$. Both superalgebras $\g_{\pm1}$ are supercommutative and
the exterior algebras $\bigwedge\g_{\pm1}$ are naturally embedded in $U(\g)$. Moreover, $U(\g)=\bigwedge(\g_{-1})U(\g_0)\bigwedge(\g_1)$. As $\ad
\g_0$-modules, $\g_{\pm1}$ are irreducible and dual to one another. 
The above $\Z$-grading can be extended to a $\Z$-grading on $U(\g)$. Moreover, if $r=\mbox{dim}\g_1$, then $U(\g)_{\pm r}=U(\g_0)\bigwedge\g^r_{\pm 1}$ and $U(\g)_s=0$ for $|s|>r$.

Fix a triangular decomposition $\fg_0=\n^-_0\oplus\h\oplus\n^+_0$ of $\g_0$.  This corresponds to a triangular decomposition $\fg=\n^-\oplus\h\oplus\n^+$, where $\n^-=\n_0^- \oplus \g_{-1}$ and $\n^+=\n_0^+ \oplus \g_{1}$.  

Let $\Delta$ be a root system for $\g$, with a set of simple roots $\pi$, and let let $\Delta^{+}$ be the set of positive roots. Set $\Delta^{-}=-\Delta^{+}.$ 
Denote by $\Delta_{\0}$ the set of non-zero even roots of $\g$ and by $\Delta_{\1}$ the set of odd roots of $\g$. Set $\Delta_{\0}^{\pm}=\Delta_{\0}\cap\Delta^{\pm}$ and $\Delta_{\1}^{\pm}=\Delta_{\1}\cap\Delta_{\pm}.$ Also, let $\pi_{\0}$ be the set of simple even roots. 

Let $W$ be the Weyl group of $\g_0.$ Put $\rho_{\0}=\frac{1}{2}\sum_{\alpha\in\Delta_{\0}^{+}}\alpha$ as usual and set $w\cdot\lambda=w(\lambda+\rho_{\0})-\rho_{\0}$ for all $w\in W,$ $\lambda\in\h^{*}.$

For any $S\subseteq \pi,$ let $Q^{S}$ be the free abelian group generated by $S$. Set $\displaystyle Q_{+}^S=\sum_{\alpha\in S}\Z_{\geq0}\alpha$, $\Delta^{S}=Q^{S}\cap \Delta$ and $\Delta^{S}_{+}=Q^{S}\cap \Delta_{+}.$ 

Let  $\lm = \{ \psi \in (\n^{+})^* \mid \psi (\g_{1})=0, \psi([\n^{+}_{0},\n^{+}_{0}])=0 \}$ and let $0 \neq \psi\in\lm$.  Following \cite[\S2.3]{Ko1978}  and since $\psi$ is completely determined by its restriction to $\n_{0}^{+}$, we will say that $\psi$ is {\it non-singular} if $\psi|_{\g_{\alpha}}\neq0$  for all $\alpha\in\pi_{\0}$. Otherwise, we will say that $\psi$ is {\it singular}. Set $S_{\psi}=\{\alpha\in\pi\mid \psi|_{\g_{\alpha}}\neq0\}$. By the  definition of $\mathcal{L}$, we observe that $S_{\psi}\subseteq \pi_{\0}$. This implies that  $\Delta^{S_{\psi}}\subseteq \Delta_{\0}$. Let $P_{\0}=\Delta_{\0}^{+}\cup (-\Delta_{+}^{S_{\psi}})$. Then $P_{\0}$ is a parabolic subset of $\Delta_{\0}$ (cf. \cite[Chap. VI, \S1, no. 7, Prop. 20]{Bourbaki1}), and it is easy to see that $P_{\0}\cap(-P_{\0})=\Delta^{S_{\psi}}$. It follows from \cite[Chap. VIII, \S3, no. 4, Prop. 13]{Bourbaki2}  that $\displaystyle\g_{\psi}=\h\oplus\bigoplus_{\alpha\in\Delta^{S_{\psi}}}\g_{\alpha}$ is a finite-dimensional reductive Lie subalgebra of $\g_{0}$. Let $W_{\psi}$ be the Weyl group of $\g_{\psi}$. 

Let $\g_{\psi}=\fs_{\psi}\oplus Z,$ where $\fs_{\psi}:=[\g_{\psi,}\g_{\psi}]$ is semisimple and $\displaystyle Z$ is the center of $\g_{\psi}.$  Note that  \begin{equation}
Z=\{h\in\h\mid\alpha(h)=0\quad\mbox{\rm for all}\quad\alpha\in\Delta^{S_{\psi}}\}+\mathfrak{z},
\end{equation}where $\mathfrak{z}$ is the center of $\g_{0}$. Moreover, $\h=\h_{\psi}\oplus Z,$ where  $\h_{\psi}=\h\cap\fs_{\psi}.$ If $\psi$ is non-singular, then $S_{\psi}=\pi_{\0}$, $\g_{\psi}=\g_{0}$, and $Z=\mathfrak{z}$. Let  $ \displaystyle\fr_{\psi}^{\pm}=
\bigoplus_{\alpha\in\Delta_{+}\setminus {\Delta^{S_{\psi}}}_{+}}\g_{\pm\alpha}$
and $\fp_{\psi}^{\pm}=\g_{\psi}\oplus\fr_{\psi}^{\pm}.$\\

Assume that  $S_{\psi}=\{\gamma_{1},\dots,\gamma_{s}\}$ and $\pi\setminus S_{\psi}=\{\beta_{1},\dots,\beta_{k}\}$.

\begin{lem} \label{lem:rideal}
$\fr_{\psi}^{+}$ (respectively $\fr_{\psi}^{-}$) is an ideal of $\fp_{\psi}^{+}$ (respectively $\fp_{\psi}^{-}$).
\end{lem}

\begin{proof}
We only prove the lemma for $\fr_{\psi}^{+}$ since the proof for $\fr_{\psi}^{-}$ is similar. Let $\alpha\in \Delta_{+}\setminus \Delta^{S_{\psi}}_{+}$ and $\beta\in \Delta_{+}\cup(-\Delta^{S_{\psi}}_{+})$ with $\alpha+\beta\in\Delta.$ To prove that $\fr_{\psi}^{+}$ is an ideal of $\fp_{\psi}^{+}$, it suffices to show that $\alpha+\beta\in \Delta_{+}\setminus\Delta^{S_{\psi}}_{+}.$ Since $\alpha\in\Delta_{+}\setminus\Delta^{S_{\psi}}_{+}$ and $\pi=\{\gamma_{1},\ldots,\gamma_{s},\beta_{1},\ldots,\beta_{k}\},$ there exist $a_{i}, b_{j}\in\Z_{\geq0}$ such that $\displaystyle\alpha=\sum_{i=1}^{s}a_{i}\gamma_{i}+\sum_{j=1}^{k}b_{j}\beta_{j}$ with at least one $b_{j}>0$. If $\beta\in\Delta_{+},$ then $\displaystyle\beta=\sum_{i=1}^{s}c_{i}\gamma_{i}+\sum_{j=1}^{k}b^{\prime}_{j}\beta_{j}$ for some $c_{i}, b^{\prime}_{j}\in\Z_{\geq0}$. Therefore $\displaystyle\alpha+\beta=\sum_{i=1}^{s}(a_i+c_{i})\gamma_{i}+\sum_{j=1}^{k}(b_{j}+b^{\prime}_{j})\beta_{j}$.  Obviously $\alpha+\beta\in\Delta_{+}$. Moreover, $\displaystyle\alpha+\beta\notin\Delta^{S_{\psi}}_{+}$ since $b_{j}+b_{j}^{\prime}>0$ for at least one $j$. Now, suppose that $\beta\in(-\Delta^{S_{\psi}}_{+}).$ Then there exist $c_{i}^{\prime}\in\Z_{\geq0}$ not all zero so that $\displaystyle\beta=-\sum_{i=1}^{s}c_{i}^{\prime}\gamma_{i}$.  It follows that $\displaystyle\alpha+\beta=\sum_{i=1}^{s}(a_i-c_{i}^{\prime})\gamma_{i}+\sum_{j=1}^{k}b_{j}\beta_{j}.$ Since $\alpha+\beta\in \Delta$ and $b_{j}>0$ for at least one $j$, it must be $\alpha+\beta\in\Delta_{+}\setminus\Delta^{S_{\psi}}_{+}.$
\end{proof}

\bigskip

\subsection{Induced Whittaker modules} \label{section:construction of the induced modules} In this subsection we construct Whittaker $\g$-modules as modules induced from Whittaker $\g_0$-modules and study their properties.  

Let $\theta_{\psi}: Z(\g_{\psi})\to U(\h)$ be the Harish-Chandra homomorphism for $\g_{\psi}$ and let $\widetilde{\theta}_{\psi}: \h^{*}\to {\rm Max} Z(\g_{\psi})$ be the induced map on the maximal ideals. For any $\lambda\in\h^{*}$ let $\chi_{\lambda}$ be the character of $Z(\g_{\psi})$ corresponding to $\widetilde{\theta}_{\psi}(\lambda)\in {\rm Max}Z(\g_{\psi})$ (that is, an algebra homomorphism $\chi_{\lambda}: Z(\g_{\psi})\to\cc$  with  $\ker\chi_{\lambda}=\widetilde{\theta}_{\psi}(\lambda)$). 

\begin{prop}\cite[cf. Prop. 2.3]{mcdowell}\label{prop:unique simple Whittaker for g_psi}
There exists a unique (up to isomorphism) simple Whittaker $\g_{\psi}$-module $W_{\psi,\chi_{\lambda}}$ of type $\psi$ and central character $\chi_{\lambda}.$ Moreover, $W_{\psi,\chi_{\lambda}}$ contains a unique (up to scalar multiplication) Whittaker vector.
\end{prop}

\begin{proof} 
Set $\displaystyle\n_{\psi}=\bigoplus_{\alpha\in\Delta^{S_{\psi}}_{+}}\g_{\alpha}$. Then $\psi|_{\n_{\psi}}$ is non-singular \cite[cf. \S2.3]{Ko1978}. Let $\chi_{\fs}$ and $\bar{\chi}_{\lambda}$ be the restrictions of $\chi_{\lambda}$ to $Z(\fs_{\psi})$ and $\mathfrak{z}$ respectively.  By \cite[Thm 3.6.1]{Ko1978}) there exists a unique (up to isomorphism) simple Whittaker $\fs_{\psi}$-module $W_{\psi,\chi_{\fs}}$  of type $\psi$ and central character $\chi_{\fs}$. Put $zv=\bar{\chi}_{\lambda}(z)v$ for all $z\in \mathfrak{z}$ and $v\in W_{\psi,\chi_{\fs}}.$ It is easy to verify that  the resulting $\g_{\psi}$-module satisfies our claims.
\end{proof}

Recall that $\fp_{\psi}^{+}=\g_{\psi}\oplus\fr_{\psi}^{+}$ and extend the $\g_{\psi}$-action on $W_{\psi,\chi_{\lambda}}$ to an action of $\fp_{\psi}^{+}$ by letting $\fr_{\psi}^{+}$ act by 0. (This action is well-defined by Lemma \ref{lem:rideal}.) Let 
\begin{equation}\label{defn:the induced modules}
\widetilde{M}_{\psi,\lambda}=U(\g)\otimes _{U(\fp_{\psi}^{+})}W_{\psi,\chi_{\lambda}}.
\end{equation} Further, if $w_{\psi,\lambda}$ is a cyclic Whittaker vector of $W_{\psi,\chi_{\lambda}}$ (unique up to scalar multiplication, see Proposition \ref{prop:unique simple Whittaker for g_psi}), we set $\tilde{w}_{\psi,\lambda}=1\otimes w_{\psi,\lambda}$. Then $\widetilde{M}_{\psi,\lambda}$ is a Whittaker module in the category $\W(\g,\n)$ of type $\psi.$

Consider $U(\fr_{\psi}^{-})$ with the adjoint action of $\g_{\psi}$. It is easy to see that the map  $f$ defined by $u\otimes y\to uy$ is a $\g_{\psi}$-isomorphism of $U(\fr_{\psi}^{-})\otimes_{\cc}W_{\psi,\chi_{\lambda}}$ onto $\widetilde{M}_{\psi,\lambda}.$ Therefore
\begin{equation}\label{eqn:basic property of induced for g}\widetilde{M}_{\psi,\lambda}\cong U(\fr_{\psi}^{-})\otimes_{\cc}W_{\psi,\chi_{\lambda}}\end{equation}  as $\g_{\psi}$-modules.
\bigskip

In what follows, we prove that $\widetilde{M}_{\psi,\lambda}$ has finite length as a $\g$-module. First, we establish some notation. We denote again by $\psi$ the restriction of $\psi$ to $\n_{0}^{+}.$  Recall that $P_{\0}=\Delta_{\0}^{+}\cup(-\Delta_{+}^{S_{\psi}})$. 
Let $\displaystyle\fm_{\psi}^{\pm}=\bigoplus_{\alpha\in \Delta_{\0}^{+}\setminus\Delta^{S_{\psi}}_{+}}\g_{\pm\alpha}.$ Then $\g_{0}=\fm_{\psi}^{-}\oplus\g_{\psi}\oplus\fm_{\psi}^{+}$, where $\g_{\psi}$ is defined as before. It follows by \cite[Chap. VIII, \S3, no. 4, Prop. 11 and Def. 2]{Bourbaki2}  that $\fp_{\psi}^{0}=\g_{\psi}\oplus\fm_{\psi}^{+}$ is a parabolic subalgebra of $\g_{0}.$ Moreover, $\fr_{\psi}^{\pm}=\fm_{\psi}^{\pm}\oplus\g_{\pm1}$. Let $\lambda\in\h^{*}$ and let $W_{\psi,\chi_{\lambda}}$ be as before.  We extend the $\g_{\psi}$-action on $W_{\psi,\chi_{\lambda}}$ to an action of $\fp_{\psi}^{0}$ by letting $\fm_{\psi}^{+}$ act by 0 and we set \begin{equation}\label{eqn:induced for g_0}M_{\psi,\lambda}=U(\g_{0})\otimes _{U(\fm_{\psi}^{+})}W_{\psi,\chi_{\lambda}}.\end{equation} 
These modules were constructed and studied in  \cite{mcdowell} and \cite{milicic-soergel}. Although this construction was originally accomplished for semisimple Lie algebras, it can be extended to reductive Lie algebras and the results that  we use remain valid over $\g_{0}$.  It is obvious that if $\psi$ is non-singular, then $\g_{\psi}=\g_{0}$, $\fm_{\psi}^{\pm}=0$ and $M_{\psi,\lambda}=W_{\psi,\chi_{\lambda}}$. It follows by  \cite[Prop. 2.4 (d)]{mcdowell} that \begin{equation}\label{eqn: basic property of induced for g_0}M_{\psi,\lambda}\cong U(\fm_{\psi}^{-})\otimes_{\cc}W_{\psi,\chi_{\lambda}}\end{equation} as $\g_{\psi}$-modules.

By  (\ref{eqn:basic property of induced for g}), the fact that $\fr_{\psi}^{-}=\g_{-1}\oplus\fm_{\psi}^{-}$, the  PBW Theorem, and (\ref{eqn: basic property of induced for g_0}), it follows that \begin{equation}\label{eqn:isomorphism as g_0 modules between the two constructions}\widetilde{M}_{\psi,\lambda}\cong \bigwedge\g_{-1}\otimes_{\cc}M_{\psi,\lambda}.\end{equation} It is easy to see that this is an isomorphism of $\g_{0}$-modules, where we consider $\bigwedge\g_{-1}$ as a $\g_{0}$-module with the adjoint action.
Let $P\subset\h^*$ be the set of weights of $\bigwedge\fg_{-1}$ as a $\g_0$-module (counted with multiplicities).

\begin{prop}\label{prop: filtration of standard Whittaker as g_0-module}
As a $\g_{0}$-module $\widetilde{M}_{\psi,\lambda}$  has a filtration with subquotients $M_{\psi,\lambda+\nu}, \nu\in P$.
\end{prop}

\begin{proof}   The proof follows from  (\ref{eqn:isomorphism as g_0 modules between the two constructions}) and \cite[Lemma 5.12]{milicic-soergel}. \end{proof}

\begin{cor} As a $\g_{0}$-module $\widetilde{M}_{\psi,\lambda}$ has finite-length.
\end{cor}

\begin{proof} The proof follows from Proposition \ref{prop: filtration of standard Whittaker as g_0-module} and the fact that each $M_{\psi,\mu},\mu\in\h^{*}$ has finite length by \cite[Theorem 2.6]{milicic-soergel}.
\end{proof}



\subsubsection{} In this subsection, let $\psi\in\lm,$ $\psi\neq0.$  We will further assume that $\psi$ is singular if $\g=\fpsl(n,n)$. Under these assumptions $Z\neq0$.

Recall that $S_{\psi}=\{\gamma_{1},\dots,\gamma_{s}\}$ and $\pi\setminus S_{\psi}=\{\beta_{1},\dots,\beta_{k}\}$.

\begin{lem}\label{easy lemma-Section 4}
$\beta_{j}(h)=0$ for all $h\in\h_{\psi}$, $j=1,\dots,k$.
\end{lem}

\begin{proof} This is an easy computation using the non-degenerate even invariant supersymmetric normalized form $(\cdot,\cdot)$ on $\g$ and we omit it.\end{proof}
\bigskip

\indent For any $\lambda\in\h^{*}$, let $\bar\lambda=\lambda|_{Z}$. Note that $\bar{\gamma_{i}}=0$ for each $i$ because $\gamma_{i}(\zeta)=0$ for all
$\zeta\in Z.$

\begin{lem} \label{lem:basis of Z^*}
\begin{itemize}
\item[\rm(a)]The set $\{\bar{\beta}_{1},\dots,\bar{\beta}_{k}\}$ is a linearly independent subset of the dual space $Z^{*}$ of $Z$.
\item[\rm(b)]If $\gamma\in \Delta_{+}\setminus\Delta^{S_{\psi}}_{+}$, then $\displaystyle{\bar\gamma\in\bigoplus_{i=1}^{k}\left(\Z_{\geq0}\bar\beta_{i}\right)}$.
\end{itemize}
\end{lem}
\begin{proof} (a) Suppose that there exist $\kappa_{1},\dots,\kappa_{m}\in\cc$ such that
\begin{equation}\label{lem:basis of Z^*,eq1}
\sum_{i=1}^{m}\kappa_{i}\bar{\beta}_{i}=0 .
\end{equation}We wish to show that $\kappa_{i}=0$ for all $i$. We claim that $\displaystyle\sum_{i=1}^{m}\kappa_{i}\beta_{i}=0.$ By
(\ref{lem:basis of Z^*,eq1})  it follows that $\displaystyle\sum_{i=1}^{m}\kappa_{i}\beta_{i}(\zeta)=\sum_{i=1}^{m}\kappa_{i}\bar{\beta}_{i}(\zeta)=0$
for all $\zeta\in Z$. Since $\h=\h_{\psi}\oplus Z$, we only need to show that $\displaystyle\sum_{i=1}^{m}\kappa_{i}\beta_{i}=0$ on $\h_{\psi}$.
However this follows from Lemma \ref{easy lemma-Section 4}. Hence $\displaystyle\sum_{i=1}^{m}\kappa_{i}\beta_{i}=0$ on $\h$ and consequently
$\kappa_{i}=0$ for each $i$ by the linear independence of $\beta_{1},\dots,\beta_{k}$.

\par (b) Since $\gamma\in \Delta_{+}\setminus\Delta^{S_{\psi}}_{+}$, there exist unique
$\nu_{i},\kappa_{j}\in\Z_{\geq0}$, $i=1,\dots,s$, $j=1,\dots,k$ with $\kappa_{j}>0$ for at least one $j$ such that
$\gamma=\nu_{1}\gamma_{1}+\dots+\nu_{s}\gamma_{s}+\kappa_{1}\beta_{1}+\dots+\kappa_{k}\beta_{k}.$ Since $\bar{\gamma}_{i}=0$ for all $i$, we
must have $\bar{\gamma}=\kappa_{1}\bar{\beta}_{1}+\dots+\kappa_{k}\bar{\beta}_{k},$ which proves (ii). 

\end{proof}

\bigskip

\indent We define a partial order on $Z^{*}$ as follows: if $\alpha,\beta\in Z^{*}$, then $\alpha\leq\beta$ if and only if
 $\displaystyle\beta-\alpha\in\bigoplus_{i=1}^{k}\Z_{\geq0}\bar\beta_{i}$.

\par For any $\g_{\psi}$-module $V$ and any $\lambda\in Z^{*}$, set
\begin{equation*}
V^{\lambda}=\{\upsilon\in{V}\mid \zeta\upsilon=\lambda(\zeta)\upsilon\textrm{ for all }\zeta\in Z\}.
\end{equation*}

\bigskip


\begin{lem}\label{lem:Z acts semisimply}
Let $\g_{\psi}$ act on $U(\fr_{\psi}^{-})$ by the adjoint action. Then 
\begin{itemize}
\item [(a)] $Z$ acts semisimply on $U(\fr_{\psi}^{-})$ and $$\displaystyle U(\fr_{\psi}^{-})=\bigoplus_{\lambda\in Z^{*},\;\lambda\geq0}U(\fr_{\psi}^{-})^{-\lambda}, \quad \fr_{\psi}^{-}U(\fr_{\psi}^{-})=\bigoplus_{\lambda\in Z^{*},\;\lambda>0}U(\fr_{\psi}^{-})^{-\lambda}$$ and $U(\fr_{\psi}^{-})^{0}=\cc 1.$
\item [(b)] $U(\fr_{\psi})^{-\lambda}$ is a finite-dimensional $\g_{\psi}$-submodule of $U(\fr_{\psi}^{-})$ for all $\lambda\geq0$ in  $Z^{*}$.
 \end{itemize}
\end{lem}

\begin{proof} (a) Recall that $\displaystyle\fr_{\psi}^{-}=\bigoplus_{\beta\in \Delta_{+}\setminus\Delta^{S_{\psi}}_{+}}\g_{-\beta}$. Assume that $(\Delta_{+}\setminus\Delta^{S_{\psi}}_{+})_{\0}=\{\xi_{1},\ldots,\xi_{p}\}$ and $(\Delta_{+}\setminus\Delta^{S_{\psi}}_{+})_{\1}=\{\eta_{1},\ldots,\eta_{q}\}$. By the PBW Theorem the set $
\{x_{1}^{s_{1}}\ldots x_{q}^{s_{q}}y_{1}^{r_{1}}\dots y_{p}^{r_{p}}\mid r_{i}\in\Z_{\geq0}, s_{j}\in\{0,1\}\}$ is a basis of $U(\fr_{\psi}^{-})$, where $y_{i}\in\g_{-\xi_{i}}$, $x_{j}\in\g_{-\eta_{j}}$ for each $i,j$. 
It easy to see that $x_{1}^{s_{1}}\ldots x_{q}^{s_{q}}y_{1}^{r_{1}}\dots y_{p}^{r_{p}}\in U(\fr_{\psi}^{-})^{-\lambda}$ for
$\displaystyle\lambda=\sum_{i=1}^{p}r_{i}\bar{\xi}_{i}+\sum_{j=1}^{q}s_{j}\bar{\eta_{j}}$. Moreover,  $\lambda\geq0$ since  by Lemma \ref{lem:basis of Z^*}(b) $\bar{\xi}_{i}>0$ and $\bar{\eta_{j}}>0$  for all $i,j$ in the partial order on $Z^{*}$ defined above and
$r_{i}, s_{j}\in\Z_{\geq0}$ for all $i,j$. Clearly $ U(\fr_{\psi}^{-})^{0}$ consists of the scalar multiples of $1$ in $ U(\fr_{\psi}^{-})$.

\par(b)  If $\lambda\geq0$ is in $Z^{*}$, then  by \ref{lem:basis of Z^*}(b) there exist unique
$\kappa_{i}\in\Z_{\geq0}$ such that $\displaystyle\lambda=\sum_{i=1}^{k}\kappa_{i}\bar{\beta}_{i}$. Set
$\displaystyle|\lambda|=\sum_{i=1}^{k}\kappa_{i}$. Then $|\lambda|\in\Z_{\geq0}$ and
$|\lambda_{1}+\lambda_{2}|=|\lambda_{1}|+|\lambda_{2}|$ if $\lambda_{1},\lambda_{2}\geq0$. 
As in (a), write $(\Delta_{+}\setminus\Delta^{S_{\psi}})_{\0}=\{\xi_{1},\ldots,\xi_{p}\}$ and $(\Delta_{+}\setminus\Delta^{S_{\psi}})_{\1}=\{\eta_{1},\ldots,\eta_{q}\}$. 
Since $|\bar{\xi}_{i}|, |\bar{\eta}_{j}|>0$ for all $i,j$, we have that for a given $m\in\Z_{>0}$, $|\bar{\xi}_{i}|+|\bar{\eta}_{j}|=m$ for only finitely many $i,j$. 
If $U(\fr_{\psi}^{-})^{-\lambda}\neq0$, it follows from the proof of (a) that $\displaystyle\lambda=\sum_{i=1}^{p}r_{i}\bar{\xi}_{i}+\sum_{j=1}^{q}s_{j}\bar{\eta_{j}},$ for $r_{i}\in\Z_{\geq0}$  and $s_{j}\in\{0,1\}$. Consequently $\displaystyle|\lambda|=\sum_{i=1}^{p}r_{i}|\bar{\xi}_{i}|+\sum_{j=1}^{q}s_{j}|\bar{\eta}_{j}|$. Since the number of such expressions for $|\lambda|$ is finite, $U(\fr_{\psi}^{-})^{-\lambda}$ must be finite-dimensional. 

\end{proof}






\bigskip

\begin{prop}\label{prop:basic properties of the induced modules}

 \begin{itemize}
\item [(a)] $Z$ acts semisimply on $\widetilde{M}_{\psi,\lambda}$ and
 $$\widetilde{M}_{\psi,\lambda}=\bigoplus_{\mu\in Z^{*},\;\mu\geq0}\widetilde{M}_{\psi,\lambda}^{\overline{\lambda}-\mu},$$  
 where 
 $$\widetilde{M}_{\psi,\lambda}^{\overline{\lambda}-\mu}\cong U(\fr_{\psi}^{-})^{-\mu}\otimes_{\cc} W_{\psi,\chi_{\lambda}}$$
 for all $\mu\in Z^{*},\mu\geq0$ as $\g_{\psi}$-modules. In particular,  $\widetilde{M}^{\overline{\lambda}}_{\psi,\lambda}\cong W_{\psi,\chi_{\lambda}}.$
\item [(b)] For all $\mu$, $\widetilde{M}^{\overline{\lambda}-\mu}_{\psi,\lambda}$ has finite length as a $\g_{\psi}$-module.
\item [(c)]  $\widetilde{M}_{\psi,\lambda}\cong \widetilde{M}_{\psi,\mu}$ if and only if $W_{\psi}\cdot\lambda=W_{\psi}\cdot\mu.$
\end{itemize}
\end{prop} 

\begin{proof}


\par (a) By Lemma \ref{lem:Z acts semisimply}, $Z$ acts semisimply on $ U(\fr^{-}_{\psi})$ via the adjoint action and
$\displaystyle U(\fr^{-}_{\psi})=\bigoplus_{\mu\in Z^{*},\;\mu\geq0} U(\fr^{-})^{-\mu}.$ Clearly the isomorphism $f$ of (\ref{eqn:basic property of induced for g}) maps
$ U(\fr^{-}_{\psi})^{-\mu}\otimes_{\complex}W_{\psi,\chi_{\lambda}}$ isomorphically onto $\widetilde{M}_{\psi,\lambda}^{\bar{\lambda}-\mu}$ for every $\mu\geq0$ since

$\zeta(u\otimes v)=[\zeta,u]\otimes v+u\otimes \zeta v=(\bar{\lambda}-\mu)(\zeta)(u\otimes v)$ for any $\zeta\in Z$, $u\in U(\fr^{-}_{\psi})^{-\mu}$, $v\in W_{\psi,\chi_{\lambda}}$. In particular, if
$\mu=0$, then $\widetilde{M}_{\psi,\lambda}^{\bar{\lambda}}=W_{\psi,\chi_{\lambda}}$ since $ U(\fr_{\psi}^{-})^{0}\cong\complex.$

\par (b) As an $\fs_{\psi}$-module, $U(\fr_{\psi}^{-})^{-\mu}\otimes_{\cc}W_{\psi,\chi_{\lambda}}$ is the tensor product of a finite dimensional $\fs_{\psi}$-module and the simple Whittaker $\fs_{\psi}$-module $W_{\psi,\chi_{\lambda}}.$ Consequently, $U(\fr_{\psi}^{-})^{-\mu}\otimes_{\cc}W_{\psi,\chi_{\lambda}}$ has finite length as an $\fs_{\psi}$-module by \cite[Thm 4.6]{Ko1978}, and hence as a $\g_{\psi}$-module.

\par (c) The proof is similar to that of \cite[Prop. 2.1 (1)]{milicic-soergel}. Suppose that $W_{\psi}\cdot \lambda=W_{\psi}\cdot\mu$. Then $\widetilde{\theta}_{\psi}(\lambda)=\widetilde{\theta}_{\psi}(\mu)$ and this  implies that $\chi_{\lambda}=\chi_{\mu}$. Hence $\widetilde{M}_{\psi,\lambda}=\widetilde{M}_{\psi,\mu}$. On the other hand, if $\widetilde{M}_{\psi,\lambda}\cong \widetilde{M}_{\psi,\mu},$ then by part (a), $\widetilde{M}_{\psi,\lambda}^{\bar\lambda}\cong \widetilde{M}_{\psi,\mu}^{\bar\mu}.$ Therefore, $W_{\psi,\chi_{\lambda}}\cong W_{\psi,\chi_{\mu}}$ which implies that $\chi_{\lambda}=\chi_{\mu}$ by Proposition \ref{prop:unique simple Whittaker for g_psi}. Consequently,  $W_{\psi}\cdot \lambda=W_{\psi}\cdot\mu.$ This completes the proof. \end{proof}

\begin{prop} \label{prop:proper submodules of the induced modules}
A $\g$-submodule of $\widetilde{M}_{\psi,\lambda}$ is proper if and only if it is contained in $\displaystyle\bigoplus_{\mu\in Z^{*},\;\mu>0}\widetilde{M}_{\psi,\lambda}^{\overline{\lambda}-\mu}.$ Moreover, $\widetilde{M}_{\psi,\lambda}$ has a unique maximal submodule and a unique irreducible  quotient $\widetilde{L}_{\psi,\lambda}$.
\end{prop}

\begin{proof}It  is a standard fact about $Z$-weight modules that every $\g$-submodule $N$ of $\widetilde{M}_{\psi,\lambda}$ has a decomposition $\displaystyle N=\bigoplus_{\mu\in Z^{*},\;\mu\geq0}(N\cap\widetilde{M}_{\psi,\lambda}^{\overline{\lambda}-\mu})$. The first statement follows from this decomposition and the fact that
$\widetilde{M}_{\psi,\lambda}^{\bar{\lambda}}\cong W_{\psi,\chi_{\lambda}}$ is a simple $\g_{\psi}$-module. Let $M$ be the sum of all proper $ U(\g)$-submodules of $\widetilde{M}_{\psi,\lambda}$. It follows by the above fact that $\displaystyle M\subseteq\bigoplus_{\mu\in Z^{*},\;\mu>0}\widetilde{M}_{\psi,\lambda}^{\bar{\lambda}-\mu}.$ Thus, $M$ is a proper $ U(\g)$-submodule
of $\widetilde{M}_{\psi,\lambda}$. As $M$ contains all proper submodules, it is the unique maximal submodule of $\widetilde{M}_{\psi,\lambda}$, and
$\widetilde{L}_{\psi,\lambda}:=\widetilde{M}_{\psi,\lambda}/M$ is the unique irreducible quotient of $\widetilde{M}_{\psi,\lambda}.$
\end{proof}

By Proposition \ref{prop:proper submodules of the induced modules} and Proposition \ref{prop:basic properties of the induced modules}(c), it follows that
 $\widetilde{L}_{\psi\lambda}\cong \widetilde{L}_{\psi,\mu}$ if and only if $W_{\psi}\cdot\lambda=W_{\psi}\cdot\mu.$

We now wish to determine when $\widetilde{M}_{\psi,\lambda}$ contains proper submodules. The following corollary follows easily from Proposition \ref{prop:proper submodules of the induced modules}.

\begin{cor}\label{cor:proper submodules-maximal vectors}$\widetilde{M}_{\psi,\lambda}$ contains a proper submodule if and only if there exists $\mu\in Z^{*},$ $\mu>0$ and $0 \neq v \in \widetilde{M}_{\psi,\lambda}^{\bar{\lambda}-\mu}$ such that $\fr_{\psi}^{+}v=0.$\end{cor}
In this case we call $v$ a {\it maximal vector}. 

\begin{prop}\label{prop:criterion for irreducibility for the induced modules}
$\widetilde{M}_{\psi,\lambda}$ is simple if and only if $\widetilde{M}_{\psi,\lambda}$ contains a unique (up to scalar) cyclic Whittaker vector.
\end{prop}

\begin{proof} If ${\rm dim} _{\cc}{\rm Wh}_{\psi} \widetilde{M}_{\psi,\lambda}=1,$ then $\widetilde{M}_{\psi,\lambda}$ is simple by  Corollary \ref{cor:Whittaker module with 1-dim'l space of Whittaker vectors is irreducible}. 

Suppose that $\widetilde{M}_{\psi,\lambda}$ is simple and let $v\in \Wh_{\psi}(\widetilde{M}_{\psi,\lambda}).$ Clearly, $v$ is a maximal vector of $\widetilde{M}_{\psi,\lambda}$. Since $\fr_{\psi}^{+}\widetilde{M}_{\psi,\lambda}^{\bar{\lambda}}=0$ we have that either $v\in \widetilde{M}^{\bar{\lambda}}_{\psi,\lambda}$ or $\displaystyle v\in\bigoplus_{\mu>0} \widetilde{M}_{\psi,\lambda}^{\bar{\lambda}-\mu}.$ Now suppose that $\displaystyle v\in\bigoplus_{\mu>0} \widetilde{M}_{\psi,\lambda}^{\bar{\lambda}-\mu}.$ Then,  $\displaystyle N=U(\g)v=U(\fr_{\psi}^{-})U(\g_{\psi})v\subseteq \bigoplus_{\mu>0} \widetilde{M}_{\psi,\lambda}^{\bar{\lambda}-\mu}$ by Lemma \ref{lem:Z acts semisimply}(a). Then, by Proposition \ref{prop:proper submodules of the induced modules}, $N$ is a proper submodule of $\widetilde{M}_{\psi,\lambda}$ which is a contradiction.  Hence, $v\in \widetilde{M}_{\psi,\lambda}^{\bar{\lambda}}.$ But $\widetilde{M}_{\psi,\lambda}^{\bar{\lambda}}\cong W_{\psi,\chi_{\lambda}}$ as $\g_{\psi}$-modules by \ref{prop:basic properties of the induced modules}(a). By Proposition \ref{prop:unique simple Whittaker for g_psi}, it follows that $v=k \tilde{w}_{\psi,\lambda}$ for some $k\in\cc$, which completes the proof. \end{proof}


\subsubsection{} Let $\psi\in\lm,$ $\psi\neq0$ and we continue to  assume that $\psi$ is singular if $\g=\fpsl(n,n)$. We show that the modules $\widetilde{M}_{\psi,\lambda}$ are not in general simple. Recall the definition of $M_{\psi,\lambda}$ in (\ref{eqn:induced for g_0}).

\begin{prop}\label{prop:weight space decomposition of induced for g_0} \cite[Prop. 2.4]{mcdowell} $Z$ acts semisimply on $M_{\psi,\lambda}$, and  
\begin{equation}\label{eqn:decomposition of M_psi,lambda}\displaystyle M_{\psi,\lambda}=\bigoplus_{\nu\in Z^{*},\;\nu\geq0}M_{\psi,\lambda}^{\bar\lambda-\nu},\end{equation} where  \begin{equation}\label{eqn:weightspaces of induced for g_0}M_{\psi,\lambda}^{\bar\lambda-\nu}\cong U(\fm_{\psi}^{-})^{-\nu}\otimes_{\cc} W_{\psi,\chi_{\lambda}}\end{equation} as $\g_{\psi}$-modules for each $\nu>0$ in $Z^{*}$ and $M_{\psi,\lambda}^{\bar{\lambda}}\cong W_{\psi,\chi_{\lambda}}.$
\end{prop}

\par The following Lemma can be easily deduced from Lemma 2.12 in \cite{mcdowell}.

\begin{lem}\label{lem:M_psi,lambda not simple criterion}
$M_{\psi,\lambda}$ contains a proper submodule if and only if there exists $\nu>0$ in $Z^{*}$ and $\displaystyle\hat{v}_{0}\in M_{\psi,\lambda}^{\bar{\lambda}-\nu}$ such that $\fm_{\psi}^{+}\hat {v}_{0}=0.$
\end{lem}


\begin{prop}\label{prop: the induced modules are not in general simple}
Suppose that  $M_{\psi,\lambda}$ is not simple. Then $\widetilde{M}_{\psi,\lambda}$ is not simple.
\end{prop}

\begin{proof}
Since $M_{\psi,\lambda}$ is not simple, by Lemma \ref{lem:M_psi,lambda not simple criterion} there exists $\displaystyle\hat {v}_{0}\in  M_{\psi,\lambda}^{\bar{\lambda}-\nu},$ $\nu>0$ such that $\fm_{\psi}^{+}\hat {v}_{0}=0.$ By (\ref{eqn:isomorphism as g_0 modules between the two constructions}), we may identify $\hat{v}_{0}$ with $\hat{v}=1\otimes \hat{v}_{0}$ in $\widetilde{M}_{\psi,\lambda}^{\bar{\lambda}-\nu}$. Clearly \begin{equation}\label{eqn:action of m_psi on hatv} \fm_{\psi}^{+}\hat{v}=0.\end{equation}Since $[\g_{1},\fm_{\psi}^{-}]\subseteq\g_{1}$, $\g_{1}W_{\psi,\chi_{\lambda}}=0$ by definition, and $U(\fm_{\psi}^{-})$ is generated by $\fm_{\psi}^{-}$, it is easy to see that  $\g_{1}\hat{v}=0$. But $\fr_{\psi}^{+}=\fm_{\psi}^{+}\oplus\g_1$. Hence by  (\ref{eqn:action of m_psi on hatv}), we have that  $\fr_{\psi}^{+} \hat{v}=0.$  Therefore by  Corollary \ref{cor:proper submodules-maximal vectors}, $\widetilde{M}_{\psi,\lambda}$ is not simple.
\end{proof}

In \cite{mcdowell1} it was shown that if $\mathfrak{s}$ is a finite-dimensional  semisimple Lie algebra over $\cc$, then the induced $\mathfrak{s}$-modules $M_{\psi,\lambda}$ can have proper submodules.  In particular, let   $\mathfrak{s}=\mathfrak{sl}_{3}$ and following \cite{mcdowell1} let $\psi$ and $\chi_{\lambda}$  be chosen so that $M_{\psi,\lambda}$ is not simple as an $\mathfrak{sl}_3$-module. We can trivially extend the action of $\mathfrak{sl}_{3}$ to an action of $\gl_{3}$ by letting the center of this reductive Lie algebra act by zero on $M_{\psi,\lambda}$. Clearly, $M_{\psi,\lambda}$ will have proper submodules as a $\mathfrak{gl}_{3}$-module.   This implies by Proposition \ref{prop: the induced modules are not in general simple} that in the case  $\g=\mathfrak{sl}(1,3)$ the $\g$-modules $\widetilde{M}_{\psi,\lambda}$ are not simple in general.

\bigskip

\subsection{Strongly typical simple Whittaker modules}\label{subsection:Strongly typical simple Whittaker modules}
For a fixed $\psi$, simple Whittaker modules for Lie algebras are in correspondence with central characters of $\g$.  In order to carry this link over to Lie superalgebras through the induced modules $\widetilde{M}_{\psi,\lambda}$, we restrict to strongly typical characters.

Let $Z(\g_0)$ be the center of $U(\g_0)$, let $\chi$ be a character of $Z(\g_{0})$ (that is, an algebra homomorphism $\chi: Z(\g_0) \rightarrow \cc$), and let $\psi\in\lm,$ $\psi\neq0.$ Recall that we denote again by $\psi$ the restriction of $\psi$ to a Lie  algebra homomorphism $\n_{0}^{+}\to\cc$. Following \cite[cf. Definition 1.5]{mcdowell}, let $K(\chi,\psi)$  be the category whose objects are the $\g_{0}$-modules $M$ such that (i) $M$ is finitely generated, (ii) for each $v\in M$ there exists a nonnegative integer $n$ such that  $(\ker \psi)^{n}v=0$, and (iii) $M$ has central character $\chi$. 

Let $\Omega:Z(\g_{\psi})\to\cc$ be a  character of $Z(\g_{\psi}).$  A vector $v$, in a $g_{0}$-module $M$, is
said to be a $(\Omega, \psi)$-vector if $xv = \psi(x) v$ for all $x \in \n_{0}^{+}$ and $zv = \Omega(z)v$ for all
$z\in Z(\g_{\psi})$.  

\begin{thm}\cite[Theorem 2.7]{mcdowell}\label{thm:mcdowell}
Let $\chi$ be a character of $Z(\g_{0})$, let $\psi\in\lm,$ $\psi\neq0$ and let $M$ be a nonzero module in $K(\chi,\psi).$ Then there exists a character $\Omega$ of $Z(\g_{\psi})$ such that $M$ contains a nonzero $(\Omega,\psi)$-vector.
\end{thm}

As the following version of Schur's Lemma states (cf. \cite[Lemma 2.1.4]{Gorelik2000}), the center $Z(\g)$ of a Lie superalgebra acts by a central character on simple modules.  

\begin{lem}\label{lem:Schur's Lemma}
Let $\g$ be finite or countable dimensional Lie superalgebra, and let $M$  be a simple $\g$-module. Then there exists a  character $\widetilde\chi$ of $Z(\g)$ such that $zv=\widetilde\chi(z)v$ for all $z\in Z(\g)$ and $v\in M.$
 \end{lem}

Let $T \in U(\g)$ be a special ghost element constructed in \cite{Gorelik2000}. Following \cite{Gorelik2_2002}, a  character $\widetilde\chi: Z(\g)\to\cc$ is {\it strongly typical} if $T^{2}\not\in \ker\widetilde{\chi}.$ We call a simple Whittaker $\g$-module $V$ {\it strongly typical} if  $V$ has a strongly typical central character $\widetilde{\chi}.$ 

Analogous to (\ref{defn:space of Whittaker vectors for psi}),  define
$$
\Wh_{\psi}^{0}(V)=\{v\in V\mid xv=\psi (x) v \textrm{ for all } x\in\n_{0}^+ \}$$
for any $\g$-module $V$. Note that $\Wh_{\psi}(V)\subseteq \Wh_{\psi}^{0}(V)$.

\begin{thm}\label{thm: irreducible Whittaker homomorphic image of induced}
Let $\psi\in\lm$, $\psi\neq0$ and let $V$ be a strongly typical simple Whittaker $\g$-module of type $\psi$. Then there exists $\lambda\in \h^{*}$ such that $V$ is a homomorphic image of $\widetilde{M}_{\psi,\lambda}$ as $\g$-modules.
\end{thm}

\begin{proof}
Suppose that $v$ is a cyclic Whittaker vector of $V$. Since $V$ is  a strongly typical simple  $\g$-module,  $V$  has a strongly typical central character $\widetilde{\chi}.$  

Let $u\in Z(\g_0)$ be a generator of $Z(\g_0)$.  We claim that  $\textrm{span}_\cc\{u^{k}v:k\geq0\}$ is finite dimensional.  Theorem 2.5 in \cite{Musson1997} (see also Lemma 8.3.1 in \cite{Gorelik1_2002} for a correction of a misprint) implies that there exist $z_0,z_1,\ldots,z_l\in Z(\g)$ (where $l$ is the number of elements in the set of sums of distinct odd positive roots) such that $\displaystyle\sum_{i=0}^{l} u^{i}z_{i}=0$ and $z_{l}=T^{2}$. Therefore $u^{l}\widetilde{\chi}(z_{l})v+\ldots+u\widetilde{\chi}(z_1)v+\widetilde{\chi}(z_0)v=0$, where $\widetilde{\chi}(z_{l})=\widetilde{\chi}(T^{2})\neq0$ since $\widetilde{\chi}$ is a strongly typical central character. It follows that $\textrm{span}_\cc\{u^{k}v:k\geq0\}$ is finite dimensional. 

Since $Z(\g_0)$ is a finitely generated polynomial algebra,  this implies that $Z(\g_0)v$ is a finite dimensional vector space on which  $Z(\g_0)$ acts as a commuting set of endomorphisms. Therefore $Z(\g_0)v$ contains a common eigenvector $w$ for the action of $Z(\g_0)$. That is, there exists an algebra homomorphism $\chi:Z(\g_{0})\to\cc$ such that $zw=\chi(z)w$ for all $z\in Z(\g_0)$. Moreover, $xw=\psi(x)w$ for all $x\in\n^+_0$ since $w\in Z(\g_0)v$ and $v\in \Wh_{\psi}^{0}(V)$. Furthermore, as $[\g_{1},\g_{0}]\subseteq \g_{1}$, $U(\g_{0})$ is generated by $\g_{0}$, and $\g_1 v=0$, it follows that $\g_1w=0$. Hence $xw=\psi(x)w$ for all $x\in\n^+$. 

Let $M=U(\g_0)w.$ Clearly $M$ is in $K(\chi,\psi)$. By Theorem \ref{thm:mcdowell}, it follows that there exists $\lambda\in\h^{*}$ such that  $M$ contains a $(\chi_{\lambda},\psi)$-vector, say $\widetilde{w}.$ In particular, $\widetilde{w}\in {\rm Wh}_{\psi}^{0}(M)$, hence $\fm_{\psi}^{+}\widetilde{w}=0$ by definition. 
Since $[\g_{1},\g_{0}]\subseteq \g_{1}$, $U(\g_{0})$ is generated by $\g_{0}$, and $\g_{1}w=0$, it follows that $\g_1 \tilde{w}=0$.  
Because  $\fr_{\psi}^{+}=\fm_{\psi}^{+}\oplus\g_{1}$, we then have that $\fr_{\psi}^{+}\widetilde{w}=0.$ 

Let $W=U(\g_{\psi})\widetilde{w}.$ Then $W$ is a Whittaker $\g_{\psi}$-module of type $\psi$ and character $\chi_{\lambda}$, hence $W$ is simple and isomorphic to $W_{\psi,\chi_{\lambda}} $ by Proposition \ref{prop:unique simple Whittaker for g_psi}.  Further, the simplicity of $V$ implies that $V=U(\g)\widetilde{w}$. Hence, the map $\widetilde{M}_{\psi,\lambda}\to V,$ defined by $u\tilde{w}_{\psi,\lambda}\to u\widetilde{w}$ defines a surjective homomorphism of $U(\g)$-modules. 
\end{proof}

\begin{cor}Let $\psi\in\lm$, $\psi\neq0$ such that $\psi$ is singular if $\g=\fpsl(n,n)$. Then the $\{\widetilde{L}_{\psi,\lambda}\mid \lambda\in\h^{*}/W_{\psi}^{\cdot}\}$ describe the isomorphism classes of the strongly typical simple Whittaker $\g$-modules. 
\end{cor}

\section{The case $\psi$  is non-singular}\label{section:nonsingular}

  For this section, we fix  $0 \neq \psi\in\lm$  non-singular and  we describe the space ${\rm Wh}_{\psi}^{0}(\widetilde{M}_{\psi,\lambda})$. We use this description to  show that in the case $\g=\fsl(1,2)$,  the induced modules $\widetilde{M}_{\psi,\lambda}$ are  in fact simple for any $\psi\in\lm$ and $\lambda\in\h^{*}.$

Recall that $P\subset\h^*$ the set of weights of $\bigwedge\fg_{-1}$ as a $\g_0$-module (counted with multiplicities).

\begin{prop}\label{prop:dimension of space of psi_0 vectors in induced}
Assume that $0 \neq \psi \in \lm$ is non-singular. As a $\g_{0}$-module   $\M_{\psi,\lambda}$ has a composition series with length equal to $\rm {dim }\bigwedge \fg_{-1}$.  The composition factors are isomorphic to the irreducible Whittaker $\g_{0}$-modules $W_{\psi,\chi_{\lambda+\nu}}$, where $\nu\in P$. Furthermore the space of Whittaker vectors $\Wh_{\psi}^0(\widetilde{M}_{\psi,\lambda})$ is finite-dimensional and \begin{equation}\hbox {\rm dim } \Wh_{\psi}^{0}(\M_{\psi,\lambda})=\hbox{ \rm dim }\bigwedge \fg_{-1}.\end{equation}\end{prop}

\begin{proof} Since $\psi$ is non-singular, it follows that  $\g_{\psi}=\g_{0}$ and $M_{\psi,\lambda}=W_{\psi,\chi_{\lambda}}$. Therefore  $\M_{\psi,\lambda}\cong \bigwedge\fg_{-1}\otimes_{\cc}W_{\psi,\chi_{\lambda}}$ by (\ref{eqn:isomorphism as g_0 modules between the two constructions}) as $\g_{0}$-modules. 
Since $\bigwedge \g_{-1}$ is finite-dimensional, the result follows from \cite[Thm 4.6]{Ko1978}.
\end{proof}

Since $\Wh_{\psi}(\M_{\psi,\lambda})\subseteq\Wh_{\psi}^{0}(\M_{\psi,\lambda})$, we obtain the following:

\begin{cor}Assume that $\psi\in\lm$ is non-singular. Then the space $\Wh_{\psi}(\widetilde{M}_{\psi,\lambda})$ is finite-dimensional.
\end{cor}

\subsection{Classification of simple Whittaker modules for $\fsl(1,2)$}
In this section, we consider more carefully the modules $\M_{\psi,\lambda}$ for $\g=\fsl(1,2)$.  In Section \ref{section:construction of the induced modules}, we showed that these modules always have finite composition length.  Here we use Corollary \ref{cor:Whittaker module with 1-dim'l space of Whittaker vectors is irreducible} show that for this superalgebra they are in fact always simple. 
The Lie superalgebra $\fsl(1,2)$ has the following basis:
$$
h=\left(\begin{array}{ccc}0 & 0 & 0 \\0 & 1 & 0 \\0 & 0 & -1\end{array}\right), \quad z=\left(\begin{array}{ccc}1 & 0 & 0 \\0 & 1 & 0 \\0 & 0 & 0\end{array}\right)
$$
$$
x_1=\left(\begin{array}{ccc}0 & 0 & 0 \\0 & 0 & 1 \\0 & 0 & 0\end{array}\right), \quad x_2=\left(\begin{array}{ccc}0 & 1 & 0 \\0 & 0 & 0 \\0 & 0 & 0\end{array}\right), \quad x_3=\left(\begin{array}{ccc}0 & 0 & 1 \\0 & 0 & 0 \\0 & 0 & 0\end{array}\right)
$$

$$
y_1=\left(\begin{array}{ccc}0 & 0 & 0 \\0 & 0 & 0 \\0 & 1 & 0\end{array}\right), \quad y_2=\left(\begin{array}{ccc}0 & 0 & 0 \\1 & 0 & 0 \\0 & 0 & 0\end{array}\right), \quad y_3=\left(\begin{array}{ccc}0 & 0 & 0 \\0 & 0 & 0 \\1 & 0 & 0\end{array}\right)
$$

Then, $\g_0 = \mbox{span} \{h,z,x_1, y_1 \}$, $\g_1= \mbox{span} \{ x_2, x_3\}$, and $\g_{-1}= \mbox{span} \{y_2, y_3\}$.  Note that $\g_0 \cong \fsl_2\oplus\cc z$, where $\{x_1,y_1,h\}$ is an $\fsl_2$-triple. Let $C=4y_1x_1+h^{2}+2h$ be the Casimir element of $U(\fsl_2)$. One can easily check that $Z(\g_0)=\cc[C,z]$. 

This gives a triangular decomposition for $\g$, where $\h=\mbox{span} \{h, z\}$ is a Cartan subalgebra and $\n^+=\mbox{span} \{x_1, x_2, x_3 \}$. (Note that $\n^{+}_{0}=\mbox{span}\{x_1\}$ and $\n^+_{1}=\mbox{span}\{x_2, x_3\}$.) Let $\psi \in \lm$ be non-singular. Then $\psi|_{\n_1^+}=0$ and suppose that $\psi(x_1)=a\in\cc$, $a\neq 0$.   Let $\lambda \in \h^*$, and assume that  $\chi_{\lambda}(C)=b$ and $\chi_{\lambda}(z)=c$, $b,c \in \cc$. Let $W_{\psi,\chi_{\lambda}}$ be defined as before.

Clearly, $\g_{\psi}=\g_{0}$, $\fr_{\psi}^{-}=\g_{-1}$. By (\ref{eqn:basic property of induced for g})
\begin{equation}\label{eqn: isom of induced for g in example}
\widetilde{M}_{\psi,\lambda}\cong \bigwedge\g_{-1}\otimes_{\cc}W_{\psi,\chi_{\lambda}}
\end{equation} as $\g_0$-modules. In what follows, we will identify $x\otimes v$ with $xv$ for any $x\in \bigwedge\g_{-1}, v\in W_{\psi,\chi_{\lambda}}$. Then $\M_{\psi,\lambda}=U(\g)w$ is a Whittaker $\g$-module of type $\psi$ with cyclic Whittaker vector $w$.

\begin{lem}
\begin{itemize}
\item [\rm(a)]The set $\{h^{m} w\mid m\geq0\}$ is a $\cc$-basis of $W_{\psi,\chi_{\lambda}}$.
\item [\rm(b)] The set $\{y_2^k y_3^{\ell} h^{m}w\mid k,\ell\in\{0,1\}, m\geq0\}$ is a $\cc$-basis of $\widetilde{M}_{\psi,\lambda}.$
\end{itemize}
\end{lem}

\begin{proof}
Part (a) follows  from  \cite[Lemma 2]{mcdowell1}. The claim in (b) follows from part (a), (\ref{eqn: isom of induced for g in example}), and the PBW Theorem. 
\end{proof}

\begin{lem}\label{exam:lem:basis of Wh_psi_0}
The set $\{w, y_2w, y_2y_3 w, 2ay_3w+y_2hw\}$ is a basis  of $\Wh_{\psi}^{0}(\widetilde{M}_{\psi,\lambda})$.
\end{lem}

\begin{proof}
It follows by Proposition \ref{prop:dimension of space of psi_0 vectors in induced} that $\mbox{dim}\Wh_{\psi}^{0}(\M_{\psi,\lambda})=\mbox{dim}\bigwedge\g_{-1}=4$. It easy to check that  $w, y_2w, y_2y_3 w, 2ay_{3}w+y_{2}h_{1} w$ are in  $\Wh_{\psi}^{0}(\M_{\psi,\lambda})$ and they are obviously linearly independent.
\end{proof}

\begin{lem} \label{lem: The space Wh is one-dimensional}
The space $\Wh_\psi(\M_{\psi,\lambda})$ is one-dimensional.
\end{lem}

\begin{proof}

Set $w_1=w, w_2=y_2w, w_3=y_2y_3 w, w_4=2ay_3+y_2h w$. Then

\begin{eqnarray}
&&x_2 w_1=0\\
\label{x_3 on w}
&&x_3 w_1=0\\
\label{x_2 on w_2}
&&x_2w_2=\frac{1}{2}cw+\frac{1}{2}h w\\
\label{x_3 on w_2}
&&x_3w_2=aw \\
\label{x_2 on w_3}
&&x_2w_3=(\frac{1}{2}c-1)y_3w+\frac{1}{2}y_3hw-\frac{1}{4a}bw_2+\frac{1}{4}y_2hw+\frac{1}{4a}y_2h^2 w\\
\label{x_3 on w_3}
&&x_3w_3=(1-\frac{1}{2}c)w_2+\frac{1}{2}w_4\\
\label{x_2 on w_4}
&&x_2w_4=\frac{1}{2}bw+\frac{1}{2}(c-1)hw\\
\label{x_3 on w_4}
&&x_3w_4=a(c-2)w
\end{eqnarray}

Now let $v\in \Wh_\psi(\M_{\psi,\lambda})$. By Lemma \ref{exam:lem:basis of Wh_psi_0}   it follows that $v$ has a unique expression $v=d_1w+d_2w_2+d_3w_3+d_4w_4$ for some $d_i\in\cc$. Since $v$ is a Whittaker vector, we know that $x_2 v=x_3v=0$.  Then (\ref{x_3 on w}), (\ref{x_3 on w_2}), (\ref{x_3 on w_3}) and (\ref{x_3 on w_4})  imply that $0=x_3 v=ad_2w+d_3(-\frac{1}{2}c)w_2+d_3 \frac{1}{2} w_4+d_4 a(c-2)w$. Therefore $d_{2}+d_4(c-2)=0$, $d_3=0$. Suppose that $c=2$. Then $d_2=0$ and $v-d_1w_1=d_4w_4\in\Wh_{\psi}(\M_{\psi,\lambda})$ which can happen only if $d_4=0$ by (\ref{x_2 on w_4}). So in the case $c=2$, the lemma is true. Now assume that $c\neq 2$. Then $v=d_1w_1+d_2w_2+\frac{d_2}{2-c} w_4$ and (\ref{x_2 on w_2}),  (\ref{x_2 on w_4}) imply that
\begin{eqnarray*}
x_2 v&=&d_2x_2w_2+\frac{d_2}{2-c}x_2w_4\\
&=&d_2(\frac{1}{2}cw+\frac{1}{2}hw)+\frac{d_2}{2-c}(\frac{1}{2}bw+\frac{1}{2}(c-1)hw).
\end{eqnarray*}Therefore  $x_2v=0$ if and only if
$\frac{1}{2}d_2+\frac{d_2}{2(2-c)}(c-1)=0$ which can happen if and only if $d_2=0$. Hence $v=d_1w$ and this completes the proof.
\end{proof}

We now have the following:

\begin{prop}
$\M_{\psi,\lambda}$ is irreducible as a $\g$-module.
\end{prop}
\begin{proof}
This follows from Lemma \ref{lem: The space Wh is one-dimensional} and Corollary \ref{cor:Whittaker module with 1-dim'l space of Whittaker vectors is irreducible}.
\end{proof}

Combining the above result with Theorem \ref{thm: irreducible Whittaker homomorphic image of induced} we obtain the following: 

\begin{cor}
Let $\g=\fsl(1,2)$,  $\psi\in\lm$, $\psi\neq0$. The $\{\M_{\psi,\lambda}, \lambda\in\h^{*}/W^{\cdot}\}$ represent the isomorphism classes of all  strongly typical simple Whittaker $\g$-modules of type $\psi$.
\end{cor}

   \bibliographystyle{amsplain}

\end{document}